\documentclass[a4paper,11pt]{amsart}

\usepackage{amsmath, amsthm, amsfonts, amssymb}
\usepackage{mathtools}
\usepackage{enumerate}
\usepackage{tensor}
\usepackage[bookmarks=false]{hyperref} 

\usepackage{color}
\usepackage[colorinlistoftodos]{todonotes}

\title{Von Neumann Equivalence and Group Approximation Properties}
\author{Ishan Ishan}

\address{Department of Mathematics, Vanderbilt University, 1326 Stevenson Center, Nashville, TN 37240, USA}
\email{ishan.ishan@vanderbilt.edu}

\newtheorem{thm}{Theorem}[section]
\newtheorem{prop}[thm]{Proposition}
\newtheorem{cor}[thm]{Corollary}
\newtheorem{lem}[thm]{Lemma}
\theoremstyle{definition}
\newtheorem{defn}[thm]{Definition}
\newtheorem{defn/lem}[thm]{Definition/Lemma}
\newtheorem{rem}[thm]{Remark}

\newcommand{\cB}{{\mathcal B}}

\newcommand{\cH}{{\mathcal H}}

\newcommand{\cM}{{\mathcal M}}

\newcommand{\cU}{{\mathcal U}}

\newcommand{\Tr}{\operatorname{Tr}}

\newcommand{\actson}{{\, \curvearrowright \,}}
\newcommand{\aactson}[1]{{\, \curvearrowright^{#1} \,}}

\makeatletter
\DeclareRobustCommand\frownotimes{\mathbin{\mathpalette\frown@otimes\relax}}
\newcommand{\frown@otimes}[2]{%
  \vbox{
    \ialign{##\cr
      \hidewidth$\m@th#1{}_\frown$\kern-\scriptspace\hidewidth\cr
      \noalign{\nointerlineskip\kern-1pt}
      $\m@th#1\otimes$\cr
    }%
  }%
}
\makeatother

\begin{document}
\begin{abstract}
The notion of von Neumann equivalence (vNE), which encapsulates both measure equivalence and $W^*$-equivalence, was introduced recently in \cite{IPR}, where it was shown that  many analytic properties, such as amenability, property (T), the Haagerup property, and proper proximality are preserved under von Neumann equivalence. In this article, we expand on the list of properties that are stable under von Neumann equivalence, and prove that weak amenability, weak Haagerup property, and the approximation property (AP) are von Neumann equivalence invariants. In particular, we get that (AP) is stable under measure equivalence. Furthermore, our techniques give an alternate proof for vNE-invariance of the Haagerup property.
\end{abstract}

\maketitle

%%%%%%%%%%

\section{Introduction}

Two infinite countable discrete groups $\Gamma$ and $\Lambda$ are \textit{measure equivalent} if there is a $\sigma$-finite measure space $(\Omega,m)$ with a measurable, measure-preserving action of $\Gamma\times\Lambda$, so that both the actions $\Gamma\actson(\Omega,m)$ and $\Lambda\actson(\Omega,m)$ admit finite-measure fundamental domains $Y,X\subset\Omega$:
\begin{align*}
\Omega=\bigsqcup_{\gamma\in\Gamma}\gamma Y=\bigsqcup_{\lambda\in\Lambda}\lambda X.
\end{align*}
This notion was introduced by Gromov in \cite[0.5.E]{Gro91} in analogy
with the topological notion of quasi-isometry for finitely generated groups and is fundamental in modern ergodic theory, especially to the study of measured group theory and
orbit equivalence. This has been investigated intensively over the years through an entire array
of techniques and many groundbreaking results have been discovered \cite{OW80, Fu99m, Fu99o, Ki10, Ki11}.
Moreover, striking connections with classical invariants in group theory have been established
\cite{Ga10, Ga02i}.

Two groups $\Gamma$ and $\Lambda$ are \textit{W$^*$-equivalent} if they have isomorphic group von Neumann algebras, ($L\Gamma\simeq L\Lambda$). This concept plays a key role in the classification of von Neumann algebras and has been studied with great success over the last decades:  Connes' work on amenability \cite{Con76}, Jones' finite index theory \cite{Jones83}, Voiculescu's free
probability \cite{Voi95, Voi96}, and more recently Popa's deformation/rigidity theory \cite{Pop07}.

Over the years, many discoveries hinted towards the notion of W$^*$-equivalence being somewhat analogous to measure equivalence as both equivalence relations preserve many of the same ``approximation type" properties. Both equivalence relations are rather coarse,
especially in the amenable case (there is a single W$^*$-equivalence class of infinite amenable ICC countable groups, and a single measure equivalence class of infinite amenable groups); both behave well with respect to approximation properties (or lack thereof), thus preserving the Haagerup property, weak amenability or Kazdhan property (T) for groups. These similarities led Shlyakhtenko to ask whether measure equivalence implied $W^*$-equivalence in the setting of ICC groups. It was shown in \cite{ChIo11} that this is not the case, although the converse implication of whether $W^*$-equivalence implies measure equivalence is still open.  

Recently, in \cite{IPR}, the following natural generalization of measure equivalence was introduced in the non-commutative setting. A \textit{fundamental domain} for the action $\Gamma\actson^\sigma\cM$ of a discrete group $\Gamma$ on a von Neumann algebra $\cM$ is a projection $p\in\cM$ such that $\sum_{\gamma\in\Gamma}\sigma_\gamma(p)=\mathbf{1}$, where the sum converges in the strong operator topology. Two groups $\Gamma$ and $\Lambda$ are \textit{von Neumann equivalent}, denoted $\Gamma\sim_{vNE}\Lambda$, if there exists a semi-finite von Neumann algebra $\cM$ with a faithful normal semi-finite trace $\Tr$, and trace-preserving commuting actions $\Gamma\actson\cM, \Lambda\actson\cM$ such that the actions of $\Gamma$ and $\Lambda$ individually admit a finite-trace fundamental domain. Von Neumann equivalence encapsulates both measure equivalence and W$^*$-equivalence. The former is straightforward to see by restricting to the case when $\cM$ is abelian, and to see later let $\theta:L\Gamma\to L\Lambda$ be a $*$-isomorphism, and consider $\cM=\cB(\ell^2\Lambda)$, where one has a trace-preserving action $\sigma:\Gamma\times\Lambda\to\mathrm{Aut}(\cM)$ given by $\sigma_{(s,t)}(T)=\theta(\lambda_s)\rho_tT\rho^*_t\theta(\lambda_s^*)$, where $\rho:\Lambda\to\cU(\ell^2\Lambda)$ is the right regular representation of $\Lambda$, and $\lambda:\Gamma\to\cU(\ell^2\Gamma)$ is the left regular representation of $\Gamma$. One can now see that the rank-one projection onto the subspace $\mathbb{C}\delta_e$ is a common fundamental domain for the actions of both $\Gamma$ and $\Lambda$.

 By developing an analysis of fundamental domains for actions of groups on semi-finite von Neumann algebra, and combining it with the language of operator spaces and Hilbert $C^*$-modules, we introduced a general procedure in \cite{IPR} for inducing actions and unitary representations in this setting. We use this to show that many analytic properties, such as amenability, property (T), the Haagerup property, and proper proximality are preserved under von Neumann equivalence.

Invariance of the Cowling-Haagerup constant under measure equivalence was proved in \cite{Jo14}, where the author extended the result of Cowling and Zimmer \cite{CZ89} from orbit equivalence to measure equivalence. With the success of induction techniques developed in \cite{IPR} to prove stability of many approximation type properties under von Neumann equivalence, it is a natural problem to investigate the behavior of Cowling-Haagerup constant under von Neumann equivalence. In this article, we give a positive result in this direction and prove the stability of (finite) Cowling-Haagerup constant (and hence weak amenability) under von Neumann equivalence. In the process, we discover an alternate proof for the invarinace of the Haagerup property under von Neumann equivalence. We also prove that the weak Haagerup property \cite{Kn16}, and the approximation property (AP) \cite{HaKr94} are also preserved under von Neuamman equivalence. In particular, (AP) is preserved under measure equivalence, a result that has not yet appeared in the literature.

\begin{thm}
\label{main theorem}
Weak amenability, weak Haagerup property, approximation property (AP), and the Haagerup property are von Neumann equivalence invariants.
\end{thm}

In order to prove Theorem \ref{main theorem} we introduce a general induction procedure for inducing Herz-Shur multipliers via von Neumann equivalence from $\Lambda$ to $\Gamma$. Our approach is inspired by Jolissaint's work in \cite{Jo14}, where some similar results are shown in the setting of measure equivalence. 

Using Popa and Vaes' result \cite[Theorem 5.1]{PV14}, Cartan rigidity results are established in \cite[Theorem 4.17]{BIP18} for groups that are both properly proximal and weakly amenable. Combining these with Theorem \ref{main theorem} above and \cite[Theorem 1.3]{IPR}, we immediately obtain the following.

\begin{cor}
Suppose $\Gamma$ is countable discrete group which is properly proximal and weakly amenable. If $\Lambda$ is any group von Neumann equivalent to $\Gamma$, then $\Lambda$ is Cartan-rigid, i.e., $L^\infty(X,\mu)\rtimes\Lambda$ admits $L^\infty(X,\mu)$ as its unique Cartan subalgebra, up to unitary conjugacy, for any free ergodic p.m.p. action $\Lambda\actson (X,\mu)$.
\end{cor}

\textbf{Acknowledgments.} The author is grateful to Jesse Peterson for suggesting the problem and making several helpful comments and suggestions. 

\section{Preliminaries and Notations}
We setup the notations and collect some facts that will be needed in this article. For facts regarding semi-finite von Neumann algebras, we refer the reader to \cite{Ta02} for their proofs, whereas those regarding multipliers on a group can be found in \cite{CH89, Jo00, Jo14}. The identity element in a von Neumann algebra $\cM$ will be denoted by $\mathbf{1}$. 

\subsection{Semi-finite traces}
A \textit{trace} on a von Neumann algebra $\cM$ is a function $\Tr$ on the positive cone $\cM_+$ with values in the extended reals $[0,\infty]$ satisfying the following conditions:
\begin{itemize}
\item[(i)] $\Tr(x+y)=\Tr(x)+\Tr(y), ~~~x,y\in\cM_+$,
\item[(ii)] $\Tr(\alpha x)=\alpha\Tr(x), \alpha\geq 0,~~~ x\in\cM_+$,
\item[(iii)] $\Tr(x^*x)=\Tr(xx^*),~~~ x\in\cM$.
\end{itemize}
A trace $\Tr$ is said to be \textit{faithful} if $\Tr(x)>0$ for any non-zero $x\in\cM_+$, \textit{semi-finite} if for every non-zero $x\in\cM_+$ there exists a non-zero $y\in\cM_+, y\leq x$ with $\Tr(y)<\infty$, \textit{finite} if $\Tr(\mathbf{1})<\infty$, and \textit{normal} if $\Tr(\sup_ix_i)=\sup_i\Tr(x_i)$ for every bounded increasing net $\{x_i\}$ in $\cM_+$. A separable von Neumann algebra $\cM$ is \textit{semi-finite} if and only if it admits a faithful normal semi-finite trace. If $\mathcal M$ is a semi-finite von Neumann algebra with a faithful normal semi-finite trace ${\rm Tr}$, we set $\mathfrak n_{\rm Tr} = \{ x \in \mathcal M \mid {\rm Tr}(x^*x) < \infty \}$, and $\mathfrak m_{\rm Tr} = \{ \sum_{j = 1}^n x_j^* y_j \mid x_j, y_j \in \mathfrak n_{\rm Tr}, 1 \leq j \leq n \}$. Both $\mathfrak n_{\rm Tr}$ and $\mathfrak m_{\rm Tr}$ are ideals in $\mathcal M$, and the trace ${\rm Tr}$ extents to a $\mathbb C$-valued linear functional on $\mathfrak m_{\rm Tr}$, which is called the \textit{definition ideal} of ${\rm Tr}$.

We let $L^1(\mathcal M, {\rm Tr})$ denote the completion of $\mathfrak m_{\rm Tr}$ under the norm $\| a \|_1 = {\rm Tr}( | a |)$, and then the bilinear form $\mathcal M \times \mathfrak m_{\rm Tr} \ni (x, a) \mapsto {\rm Tr}(xa)$ extends to the duality between $\mathcal M$ and $L^1(\mathcal M, {\rm Tr})$ so that we may identify $L^1(\mathcal M, {\rm Tr})$ with $\mathcal M_*$. We also let $L^2(\mathcal M, {\rm Tr})$ denote the Hilbert space completion of $\mathfrak n_{\rm Tr}$ under the inner product $\langle a, b \rangle_{\Tr} = {\rm Tr}(b^* a)$. Left multiplication of $\mathcal M$ on $\mathfrak n_{\rm Tr}$ then induces a normal faithful representation of $\mathcal M$ in $\mathcal B(L^2(\mathcal M, {\rm Tr}))$, which is called the \textit{standard representation}.

\subsection{Multipliers on discrete groups and associated multiplier algebras}
\label{sec:multipliers}
Let $\Gamma$ be an infinite countable discrete group. We will denote by $c_0(\Gamma)$, the space of all complex-valued functions on $\Gamma$ vanishing at infinity, i.e., $f\in c_0(\Gamma)$ if for every $\varepsilon>0$, there exists a finite set $F\subset\Gamma$ such that $|f(s)|<\varepsilon$ for all $s\in\Gamma\setminus F$. The space of all bounded complex-valued functions on $\Gamma$ will be denoted by $\ell^\infty\Gamma$, and $c_{00}(\Gamma)$ will denote the space of all finitely supported functions on $\Gamma$. For a subset $E\subset\Gamma$, we denote the characteristic function of $E$ by $\mathbf{1}_E$.

 The \textit{Fourier-Stieltjes algebra} of $\Gamma$, denoted by $B(\Gamma)$, is the set of all coefficient functions of unitary representations of $\Gamma$, that is, for every $\varphi\in B(\Gamma)$ there exists a unitary representation $(\pi,\cH)$ of $\Gamma$ and vectors $\xi,\eta\in\cH$ such that $\varphi(s)=\langle\pi(s)\xi,\eta\rangle$ for every $s\in\Gamma$. It is a Banach algebra with respect to the norm
$$\|\varphi\|_B=\inf\|\xi\|\|\eta\|,$$
where the infimum is taken over all representations of $\varphi$ as above.

The $\textit{Fourier algebra}$ of $\Gamma$, denoted by $A[\Gamma]$, is the set of all coefficient functions associated to the left regular representation of $\Gamma$. It is the norm closure of the algebra of finitely supported functions in the algebra $B(\Gamma)$.

A \textit{Herz-Schur multiplier} on $\Gamma$ is a function $\varphi:\Gamma\to\mathbb{C}$ for which there exists a Hilbert space $\cH$ and bounded functions $\xi,\eta:\Gamma\to\cH$ such that
$$\varphi(t^{-1}s)=\langle\xi(s),\eta(t)\rangle \qquad s,t\in\Gamma.$$
The set $B_2(\Gamma)$ of all Herz-Schur multipliers on $\Gamma$ is a Banach algebra with respect to the pointwise product and to the norm
$$\|\varphi\|_{B_2}=\inf\|\xi\|_\infty\|\eta\|_\infty,$$
where the infimum is taken over all representations of $\varphi$ as above. It turns out that $B_2(\Gamma)$ is a dual space, and the predual $Q(\Gamma)$ of $B_2(\Gamma)$ is obtained by completing $\ell^1\Gamma$ in the norm
\begin{align*}
\|\varphi\|_Q=\left\lbrace\left|\sum_{s\in\Gamma}\varphi(s)u(s)\right|\mid u\in B_2(\Gamma),\|u\|_{B_2}\leq 1\right\rbrace
\end{align*}
(see, \cite{He74,DH85}).

\begin{defn} 
Let $\Gamma$ be a countable discrete group.
\begin{enumerate}
\item (\cite{CCJJV}) We say that $\Gamma$ has the \textit{Haagerup property} if there exists a net $\{\varphi_i\}$ of normalized (i.e., $\varphi_i(e)=1$ for every $i$), positive definite functions on $\Gamma$ such that $\varphi_i\in c_0(\Gamma)$ for every $i$, and $\varphi_i\to 1$ pointwise.
\item (\cite{CH89}) We say that $\Gamma$ is \textit{weakly amenable} if there exists a net $\{\varphi_i\}$ of finitely supported functions on $\Gamma$ converging pointwise to the constant function $1$, and such that $\sup_i\|\varphi_i\|_{B_2}\leq C$. The \textit{Cowling-Haagerup constant} $\Lambda_{\mathrm{cb}}(\Gamma)$ is the infimum of all constants $C$ for which such a net $\{\varphi_i\}$ exists.

\item (\cite{Kn16}) We say that $\Gamma$ has the \textit{weak Haagerup property} if there  there exists a net $\{\varphi_i\}$ in $B_2(\Gamma)\cap c_0(\Gamma)$ such that $\sup_i\|\varphi_i\|_{B_2}\leq C$ and $\varphi_i\to 1$ pointwise. The \textit{weak Haagerup constant} $\Lambda_{\mathrm{wcb}}(\Gamma)$ is the infimum of all constants $C$ for which such a net $\{\varphi_i\}$ exists.
\item (\cite{HaKr94}) We say that $\Gamma$ has the \textit{approximation property} (AP) if there exists a net $\{\varphi_i\}$ of finitely supported functions on $\Gamma$ such that $\varphi_i\to 1$ in the $\sigma(B_2(\Gamma),Q(\Gamma))$-topology.
\end{enumerate}
\end{defn}

\begin{rem}
It is easy to see that every finitely supported function on $\Gamma$ can be realized as a coefficient of the left regular representation and hence $c_{00}(\Gamma)\subset A[\Gamma]$. Therefore, by \cite[Proposition 1.1]{CH89}, $\Gamma$ is weakly amenable if and only if there exists a net $\{\varphi_i\}$ in the Fourier algebra $A[\Gamma]$ such that $\varphi_i\to 1$ pointwise and $\sup_i\|\varphi_i\|_{B_2}<\infty$.
\end{rem}

\begin{rem}
The inclusion map from $B(\Gamma)$ into $B_2(\Gamma)$ is a contraction (see \cite[Corollary 1.8]{DH85}), and so the $\sigma(B_2(\Gamma),Q(\Gamma))$-closure of any subset $E$ of $B(\Gamma)$ contains the closure of $E$ in the $B(\Gamma)$-norm. Hence $\Gamma$ has (AP) if and only if the constant function $1$ is in the $\sigma(B_2(\Gamma),Q(\Gamma))$-closure of $A[\Gamma]$ in $B_2(\Gamma)$.
\end{rem}

Since $A[\Gamma]\subset B_2(\Gamma)\cap c_0(\Gamma)$, one always has $\Lambda_{\mathrm{wcb}}(\Gamma)\leq\Lambda_{\mathrm{cb}}(\Gamma)$, and a weakly amenable group has the weak Haagerup property. Similarly, as normalized, positive definite functions are Herz-Schur multipliers of norm one, it follows that if $\Gamma$ has the Haagerup property then it has the weak Haagerup property and $\Lambda_{\mathrm{wcb}}(\Gamma)=1$. It is well known that all weakly amenable groups have (AP), and there are non-weakly amenable groups with the (AP) as well (see \cite{HaKr94}).

\subsection{Actions on semi-finite von Neumann algebras}\label{sec:actions}

If $\Lambda$ is a discrete group and $\Lambda \aactson{\sigma} \mathcal M$ is an action that preserves the trace ${\rm Tr}$, then restricted to $\mathfrak n_{\rm Tr}$, the action is isometric with respect to $\| \cdot \|_2$ and hence gives a unitary representation in $\mathcal U(L^2(\mathcal M, {\rm Tr}))$, which is called the \textit{Koopman representation} and denoted by $\sigma^0: \Lambda \to  \mathcal U(L^2(\mathcal M, {\rm Tr}))$. Note that, considering $\mathcal M \subset \mathcal B(L^2(\mathcal M, {\rm Tr}))$ via the standard representation, we have that the action $\sigma: \Lambda \to {\rm Aut}(\mathcal M, {\rm Tr})$ becomes unitarily implemented via the Koopman representation, i.e., for $x \in \mathcal M$ and $s \in \Lambda$ we have $\sigma_s(x) = \sigma_s^0 x \sigma_{s^{-1}}^0$. 

\begin{defn}[\cite{IPR}]
Let $\Lambda\actson^\sigma\cM$ be an action of a discrete group $\Lambda$ on a von Neumann algebra $\cM$. A \textit{fundamental domain} for the action is a projection $p\in\cM$ so that $\{\sigma_s(p)\}_{s\in\Lambda}$ gives a partition of unity.
\end{defn}

Note that, if $p\in\cM$ is a fundamental domain, then we obtain a normal inclusion $\theta_p:\ell^\infty\Lambda\to \cM$ by $\theta_p(f)=\sum_{s\in\Lambda}f(s)\sigma_{s}(p)$. Moreover, this embedding is equivariant with respect to the $\Lambda$-actions, where $\Lambda\actson\ell^\infty\Lambda$ is the canonical left action given by $L_\lambda(f)(s)=f(\lambda^{-1}s)$. Indeed,
\begin{align*}
\sigma_\lambda(\theta_p(f))=\sum_{s\in\Lambda}f(s)\sigma_{\lambda s}(p)=\sum_{t\in\Lambda}f(\lambda^{-1}t)\sigma_t(p)=\sum_{t\in\Lambda}(L_\lambda (f))(t)\sigma_t(p)=\theta_p(L_\lambda(f))
\end{align*}

Conversely, if $\theta:\ell^\infty\Lambda\to\cM$ is an equivariant embedding, then $\theta(\delta_e)$ gives a fundamental domain. %The following lemma contains a few straightforward observations about the embedding $\theta_p$ that we will need later.

\begin{defn}[\cite{IPR}]
Let $\Gamma$ and $\Lambda$ be countable discrete groups. A \textit{von Neumann coupling} between $\Gamma$ and $\Lambda$ consists of a semi-finite von Neumann algebra $\cM$ with a faithful normal semi-finite trace $\Tr$ and a trace-preserving action $\Gamma\times\Lambda\actson\cM$ such that there exist finite-trace fundamental domains $q$ and $p$ for the $\Gamma$- and $\Lambda$-actions, respectively. The index of the von Neumann coupling is the ratio $\Tr(q)/\Tr(p)$ and is denoted by $[\Gamma:\Lambda]_{\cM}$. We say that $\Gamma$ and $\Lambda$ are \textit{von Neumann equivalent} and write $\Gamma\sim_{vNE}\Lambda$ if there exists a von Neumann coupling between them.
\end{defn}

An element $x\in\cM$ is \textit{compact}, if for every $\varepsilon>0$, there exists a projection $p\in\cM$ such that $\|xp\|<\varepsilon$ and $\mathbf{1}-p$ is finite. If $\cM$ is a semi-finite von Neumann algebra with a faithful normal semi-finite trace $\Tr$, and $p\in\cM$ is a finite-trace projection, then the map $x\mapsto\Tr(xp)$ is weak operator topology continuous.

\begin{lem}
\label{lem:compact operators} 
 Suppose $\cM$ is a semi-finite von Neumann algebra with a faithful normal semi-finite trace $\Tr$, $x\in\cM$ is compact, and $\{p_i\}$ is a net of finite-trace projections such that $p_i\to 0$ in the weak operator topology. If $\{\Tr(p_i)\}$ is uniformly bounded, then, $\Tr(xp_i)\to 0$.
\end{lem}
\begin{proof}
Given $\varepsilon>0$, there exists a projection $q\in\cM$ such that $\|xq\|<\varepsilon$ and $\Tr(\mathbf{1}-q)<\infty$. Since $p_i\to 0$ in the weak operator topology, we get that $xp_i\to0$ in the weak operator topology and hence $\Tr(p_ix(\mathbf{1}-q))\to 0$. Moreover, we have that $\{\Tr(p_i)\}$ is uniformly bounded, say by $C$, whence it follows that
\begin{align*}
\limsup_i|\Tr(xp_i)|\leq\limsup_i(\|xq\||\Tr(p_i)|+|\Tr(p_ix(\mathbf{1}-q))|)\leq \varepsilon C.
\end{align*}
Since $\varepsilon>0$ was arbitrary, the proof is complete.
\end{proof}

\section{Inducing Herz-Schur Multipliers}

We present the proof of Theorem \ref{main theorem} in this section. The proof relies on the following analogue of Lemma 2.1 in \cite{Jo14}. 

\begin{lem}
\label{main lemma}
 Let $\cM$ be a semi-finite von Neumann algebra with a faithful normal semi-finite-trace $\Tr$ and let $\Lambda\actson^\sigma(\cM,\Tr)$ be a trace-preserving action with a finite-trace fundamental domain $p$. Suppose $\Gamma\actson(\cM,\Tr)$ is another trace-preserving action that commutes with the $\Lambda$-action. For $\varphi\in\ell^\infty\Lambda$, define $\hat{\varphi}:\Gamma\to\mathbb{C}$ by
$$\hat{\varphi}(\gamma):=\frac{1}{\Tr(p)}\Tr(\sigma_\gamma(\theta_p(\varphi))p)=\frac{1}{\Tr(p)}\Tr(\theta_p(\varphi)\sigma_{\gamma^{-1}}(p)), \qquad \gamma\in\Gamma,$$
where $\theta_p:\ell^\infty\Lambda\hookrightarrow\cM$ is the $\Lambda$-equivariant embedding.
\begin{enumerate}[$($a$)$]
\item \label{item:A}If $\varphi\in B_2(\Lambda)$ is a Herz-Schur multiplier on $\Lambda$, then $\hat{\varphi}$ is a Herz-Schur multiplier on $\Gamma$ and $\|\hat{\varphi}\|_{B_2}\leq\|\varphi\|_{B_2}$. Moreover, if $\varphi$ is positive definite, then so is $\hat{\varphi}$.

\item \label{item:B} If $\Gamma\actson(\cM,\Tr)$ is mixing, i.e., the Koopman representation $\Gamma\actson L^2(\cM,\Tr)$ is mixing, and if $\varphi\in c_0(\Lambda)$, then $\hat{\varphi}\in c_0(\Gamma)$. In particular, if $\varphi\in B_2(\Lambda)\cap c_0(\Lambda)$, then $\hat{\varphi}\in B_2(\Gamma)\cap c_0(\Gamma)$.
\end{enumerate}
\end{lem}
\begin{proof}
Since $p$ is a finite-trace fundamental domain, it follows that $\hat{\varphi}$ is well-defined and $\|\hat{\varphi}\|_\infty\leq\|\varphi\|_\infty$. Let $\xi,\eta:\Lambda\to\cH_0$ be bounded functions from $\Lambda$ into a Hilbert space $\cH_0$ such that $\varphi(t^{-1}s)=\langle\xi(s),\eta(t)\rangle,~~ s,t\in\Lambda.$ 

Let $\cH=L^2(\cM,\Tr)\overline{\otimes}\cH_0$. Note that, for $\gamma\in\Gamma$, we have
\begin{align*}
\sum_{s\in\Lambda}\|\sigma_\gamma(\sigma_s(p))p\|_2^2\|\xi(s)\|^2 \leq\|\xi\|_\infty^2\sum_{s\in\Lambda}\Tr(p\sigma_\gamma(\sigma_s(p)))=\|\xi\|_\infty^2\Tr(p)<\infty.
\end{align*}

Therefore, $\hat{\xi},\hat{\eta}:\Gamma\to L^2(\cM,\Tr)\overline{\otimes}\cH_0$ given below are well-defined.
\begin{align*}
\hat{\xi}(\gamma):=\frac{1}{\sqrt{\Tr(p)}}\sum_{s\in\Lambda}\sigma_\gamma(\sigma_s(p))p\otimes\xi(s), \qquad
\hat{\eta}(\gamma):=\frac{1}{\sqrt{\Tr(p)}}\sum_{t\in\Lambda}\sigma_\gamma(\sigma_t(p))p\otimes\eta(t).
\end{align*}
One has, for every $\gamma\in\Gamma$,
\begin{align*}
\|\hat{\xi}(\gamma)\|^2 =\frac{1}{\Tr(p)}\sum_{s\in\Lambda}\langle\xi(s),\xi(s)\rangle\Tr(\sigma_s(p)\sigma_{\gamma^{-1}}(p))\leq\frac{1}{\Tr(p)}\|\xi\|_\infty^2\Tr\left(\sum_{s\in\Lambda}\sigma_s(p)\sigma_{\gamma^{-1}}(p)\right)=\|\xi\|_\infty^2.
\end{align*}
Thus, $\|\hat{\xi}\|_\infty\leq\|\xi\|_\infty$. Similarly, $\|\hat{\eta}\|_\infty\leq\|\eta\|_\infty$.
Finally, for $\gamma_1, \gamma_2\in\Gamma$, we have
\begin{align*}
\langle\hat{\xi}(\gamma_1),\hat{\eta}(\gamma_2)\rangle&=\frac{1}{\Tr(p)}\left\langle\sum_{s\in\Lambda}\sigma_{\gamma_1}(\sigma_s(p))p\otimes\xi(s),\sum_{t\in\Lambda}\sigma_{\gamma_2}(\sigma_t(p))p\otimes\eta(t)\right\rangle\\
&=\frac{1}{\Tr(p)}\sum_{s,t\in\Lambda}\langle\xi(s),\eta(t)\rangle\Tr(p\sigma_{\gamma_2}(\sigma_t(p))\sigma_{\gamma_1}(\sigma_s(p)))\\
&=\frac{1}{\Tr(p)}\sum_{t\in\Lambda}\Tr\left(p\sigma_{\gamma_2}(\sigma_t(p))\sigma_{\gamma_1}\left(\sum_{s\in\Lambda}\varphi(t^{-1}s)\sigma_s(p)\right)\right)\\
&=\frac{1}{\Tr(p)}\sum_{t\in\Lambda}\Tr(p\sigma_{\gamma_2}(\sigma_t(p))\sigma_{\gamma_1}(\sigma_t(\theta_p(\varphi))))\\
&=\frac{1}{\Tr(p)}\sum_{t\in\Lambda}\Tr(\sigma_{\gamma_2^{-1}}(p)\sigma_t(p\sigma_{\gamma_2^{-1}\gamma_1}(\theta_p(\varphi))))\\
&=\frac{1}{\Tr(p)}\Tr\left(\sum_{t\in\Lambda}\sigma_{t^{-1}}(\sigma_{\gamma_2^{-1}}(p))p\sigma_{\gamma_2^{-1}\gamma_1}(\theta_p(\varphi))\right)\\
&=\frac{1}{\Tr(p)}\Tr(\sigma_{\gamma_2^{-1}\gamma_1}(\theta_p(\varphi))p)\\
&=\hat{\varphi}(\gamma_2^{-1}\gamma_1).
\end{align*}
Therefore, $\hat{\varphi}$ is a Herz-Schur multiplier with $\|\hat{\varphi}\|_{B_2}\leq\|\varphi\|_{B_2}$. Furthermore, if $\varphi$ is positive definite, then one can take $\eta=\xi$ and it is straightforward to see that $\hat{\varphi}$ is positive definite on $\Gamma$ as well. (In fact, let $(\pi_\varphi,\cH_\varphi,\xi_\varphi)$ be the GNS-triple associated to $\varphi$. Then, as $\varphi(s)=\langle\pi_\varphi(s)\xi_\varphi,\xi_\varphi\rangle$ for every $s\in\Lambda$, we see that the function $s\mapsto\xi(s)=\pi_\varphi(s)\xi_\varphi$ works.)

 Note that, if $\varphi\in c_0(\Lambda)$, then $\theta_p(\varphi)$ is compact. Since the action of $\Gamma$ is mixing, $\sigma_\gamma(p)\to 0$ in the weak operator topology as $\gamma\to\infty$, and (b) now follows from Lemma \ref{lem:compact operators}.
\end{proof}

\begin{prop}
\label{main propositon}
 Let $\cM$ be a semi-finite von Neumann algebra with a faithful normal semi-finite trace $\Tr$ and $\Lambda\actson^\sigma(\cM,\Tr)$ be a trace-preserving action with a finite-trace fundamental domain $p$. Suppose $\Gamma\actson(\cM,\Tr)$ is another trace-preserving action that commutes with the $\Lambda$-action. Consider the map $\Phi:\ell^\infty\Lambda\to\ell^\infty\Gamma$ defined by $\Phi(\varphi)=\hat{\varphi}$, where $\hat{\varphi}$ is defined as in Lemma \ref{main lemma}. Then $\Phi$ is a contractive linear mapping from $B_2(\Lambda)$ into $B_2(\Gamma)$, and is
 \begin{enumerate}[$($a$)$]
 \item  continuous on norm bounded sets with respect to the topology of poitwise convergence.
 \item  $\sigma(B_2(\Lambda),Q(\Lambda))$-$\sigma(B_2(\Gamma),Q(\Gamma))$ continuous.
 \end{enumerate}
\end{prop}
\begin{proof}
It is clear that $\Phi:B_2(\Lambda)\to B_2(\Gamma)$ is linear, and that it is contractive follows from Lemma \ref{main lemma}\eqref{item:A}. Moreover, $\Phi:\ell^\infty\Lambda\to\ell^\infty\Gamma$ is also a linear contraction.
\begin{enumerate}[$($a$)$]
\item Let $\varphi_i\to 0$ in $B_2(\Lambda)$ pointwise and let $\|\varphi_i\|_{B_2}<C$ for every $i$. Since $\|\cdot\|_\infty\leq\|\cdot\|_{B_2}$, after passing to a subnet if necessary, we may assume that $\varphi_i\to 0$ weak$^*$. By Lemma \ref{lem: properties of embedding} it follows that $\theta_p(\varphi_i)\to 0$ in the weak operator topology. Therefore, $\Tr(\theta_p(\varphi_i)\sigma_{\gamma^{-1}}(p))\to 0$ for every $\gamma\in\Gamma$, and hence $\hat{\varphi_i}\to 0$ pointwise.

\item Since $\Phi:B_2(\Lambda)\to B_2(\Gamma)$ is a linear contraction, the dual $\Phi^*:B_2(\Gamma)^*\to B_2(\Lambda)^*$ is continuous. Therefore, to prove that $\Phi$ is $\sigma(B_2(\Lambda),Q(\Lambda))$-$\sigma(B_2(\Gamma),Q(\Gamma))$ continuous, it suffices to show that $\Phi$ maps $Q(\Gamma)$ into $Q(\Lambda)$. To this end, notice that a similar argument as in the proof of previous part shows that $\Phi:\ell^\infty\Lambda\to\ell^\infty\Gamma$ is normal, whence it follows that the dual map $\Phi^*$ maps $\ell^1\Gamma$ into $\ell^1\Lambda$. Since $\ell^1\Gamma$ and $\ell^1\Lambda$ are dense, respectively, in $Q(\Gamma)$ and $Q(\Lambda)$, it follows that $\Phi^*(Q(\Gamma))\subset Q(\Lambda)$.
\end{enumerate}
\end{proof}

\begin{proof}[Proof of Theorem \ref{main theorem}]
Suppose $\Lambda_{\mathrm{cb}}(\Lambda)\leq C$, and choose a net $\{\varphi_i\}$ of finitely supported functions on $\Lambda$ such that $\sup_i\|\varphi_i\|_{B_2}\leq C$, and $\varphi_i\to 1$ pointwise. It follows from Lemma \ref{main lemma} that $\hat{\varphi_i}\in B_2(\Gamma)$ and $\|\hat{\varphi_i}\|_{B_2}\leq C$ for every $i$. Since each $\varphi_i$ is finitely supported, we have that $\hat{\varphi_i}$ is a coefficient of the Koopman representation $\sigma^0:\Gamma\to\cU(L^2(\cM,\Tr))$. Moreover, the existence of a fundamental domain for $\Gamma$ implies that $\sigma^0$ is a multiple of the left-regular representation \cite[Proposition 4.2]{IPR} and hence $\hat{\varphi_i}\in A[\Gamma]$ for all $i$. From Proposition \ref{main propositon}, we also have that $\hat{\varphi_i}\to 1$ pointwise. This shows that $\Lambda_{\mathrm{cb}}(\Gamma)\leq C$, and $\Gamma$ is weakly amenable.

If the net $\{\varphi_i\}$ is in $c_0(\Lambda)\cap B_2(\Lambda)$, then by Lemma \ref{main lemma}\eqref{item:B}, $\hat{\varphi_i}\in c_0(\Gamma)\cap B_2(\Gamma)$ for every $i$. Now the same argument as in the previous paragraph shows that if $\Lambda$ has the weak Haagerup property, then $\Gamma$ has the weak Haagerup property and $\Lambda_{\mathrm{wcb}}(\Gamma)\leq \Lambda_{\mathrm{wcb}}(\Lambda)$.

Note that $\theta_p(\varphi)p=\varphi(e)p$. Therefore, if $\varphi(e)=1$, then $\hat{\varphi}(e)=1$. Moreover, if $\varphi$ is a normalized positive definite function, then there exists a Hilbert space $\cH_\varphi$, a unitary representation $\pi_\varphi:\Lambda\to\cU(\cH_\varphi)$, and a unit vector $\xi_\varphi\in\cH_\varphi$ such that $\varphi(s)=\langle\pi_\varphi(s)\xi_\varphi,\xi_\varphi\rangle, s\in\Lambda$ \cite[Theorem 2.5.11]{BO08}. In particular, $\|\varphi\|_\infty\leq 1$. Therefore, the arguments in the previous paragraphs show that the Haagerup property is invariant under von Neumann equivalence.

If $\Lambda$ has (AP), then $1$ is in the $\sigma(B_2(\Lambda),Q(\Lambda))$-closure of finitely supported functions on $\Lambda$. From the first paragraph we have that if $\varphi$ is a finitely supported function on $\Lambda$, then $\hat{\varphi}\in A[\Gamma]$. Therefore, Proposition \ref{main propositon} gives that $1$ is in the $\sigma(B_2(\Gamma),Q(\Gamma))$-closure of $A[\Gamma]$ inside $B_2(\Gamma)$, whence it follows that $\Gamma$ has (AP).
\end{proof}

\bibliographystyle{amsalpha}
\bibliography{reference}

\end{document}